\documentclass[12pt]{amsart}
\usepackage{a4wide,xcolor}
\usepackage[alphabetic]{amsrefs}
\usepackage[T1,T2A]{fontenc}
\usepackage[utf8]{inputenc}
\usepackage{url,hyperref}
\usepackage{amsfonts}
\usepackage{mathtools}
\usepackage{needspace}
\usepackage[pdftex]{graphicx}
\usepackage{datetime}
\usepackage{epigraph}
\usepackage{verbatim}
\usepackage{mathtools}
\usepackage{enumerate}
\usepackage[american]{babel}
\numberwithin{equation}{section}

\newcommand{\R}{\mathbb{R}}

\newcommand{\Cm}{\mathbb{C}}
\newcommand{\intt}{\int\limits}

\renewcommand{\phi}{\varphi}

\newtheorem{Thm}{Theorem}[section]
\newtheorem{theorem}[Thm]{Theorem}

\newtheorem{lemma}[Thm]{Lemma}

\newtheorem{remark}[Thm]{Remark}

\newtheorem{definition}{Definition}

\begin{document}

  \title[A monotonicity theorem for subharmonic functions on manifolds]
  {A monotonicity theorem for subharmonic functions on manifolds}
  
  \author{Aleksei Kulikov}
  \address{Tel Aviv University, School of Mathematical Sciences, Tel Aviv, 
    69978, Israel,}
  \email{lyosha.kulikov@mail.ru}
  
  \author{Fabio Nicola}
  \address{Dipartimento di Scienze Matematiche, Politecnico di Torino, Corso 
    Duca degli Abruzzi 24, 10129 Torino, Italy.}
  \email{fabio.nicola@polito.it}
  
  \author{Joaquim Ortega-Cerdà}
  \address{Department de Matematiques i Inform\`atica, Universitat de 
  Barcelona, Barcelona, Spain and Centre de Recerca Matem\'atica, 
  Barcelona, Spain.}
  \email{jortega@ub.edu}
  
  \author{Paolo Tilli}
  \address{Dipartimento di Scienze Matematiche, Politecnico di Torino, Corso 
   Duca degli Abruzzi 24, 10129 Torino, Italy.}
  \email{paolo.tilli@polito.it}
 
 \thanks{Joaquim Ortega-Cerd\`{a} was supported in part by the Spanish 
Ministerio de Ciencia e Innovaci\'on, project PID2021-123405NB-I00. Aleksei 
Kulikov was supported by BSF Grant 2020019, ISF Grant 1288/21,
and by The Raymond and Beverly Sackler Post-Doctoral Scholarship.}
  

  \begin{abstract} We provide a sharp monotonicity theorem about the
distribution of subharmonic functions on manifolds, which can be regarded as a
new, measure theoretic form of the uncertainty principle.  As an illustration of
the scope of this result, we deduce contractivity estimates for analytic
functions on the Riemann sphere, the complex plane and the Poincar\'e disc, with
a complete description of the extremal functions, hence providing a unified and
illuminating perspective of a number of results and conjectures on this subject,
in particular on the Wehrl entropy conjecture by Lieb and Solovej. In this
connection, we completely prove that conjecture for $SU(2)$, by showing that the
corresponding extremals are only the coherent states.
 Also, we show that the above (global) estimates admit a local counterpart and in all cases we characterize also the extremal subsets, among those of fixed assigned measure.
  \end{abstract}
  \maketitle
  \section{Main results and applications}
  In the recent works of the second and fourth authors \cite{NicTil1} in the
  Euclidean case and of the first author \cite{Kulikov} in the hyperbolic case,
  a new
  method was discovered for studying the distribution of analytic functions. In
  this paper we single out key properties required for this approach to work,
  in
  the process generalizing it to  wider classes  of functions and to new
  geometries, in particular the spherical geometry.

  To state the main result, we need to introduce some notation first. Let $M$
  be a
  smooth $n$-dimensional Riemannian manifold without boundary. We assume that
  it
  satisfies an isoperimetric inequality, that is for all open sets $A\subset M$
  with compact closure and smooth boundary we have
  \begin{equation}\label{isop}
    |\partial A|_m^2 \ge H(|A|_M),
  \end{equation}
  where $|\cdot|_m$ is the $n-1$ Hausdorff measure on $M$,
$|\cdot|_M$ is
  the
  $n$-dimensional volume on $M$ associated to the metric, and $H:(0, |M|_M)\to
  (0, +\infty)$ is a $C^1$ function (if $|M|_M$ is finite, we extend it to
  $H(|M|_M) = 0$).
  \begin{theorem}\label{main-diff}
    Let $M$ be an $n$-dimensional Riemannian manifold satisfying \eqref{isop} and let
    $u:M\to
    \R$ be a Morse function on $M$, $u\in C^2(M)$, such that for all $t\in \R$ 
the
    superlevel sets $u^{-1}([t, +\infty))$ are compact and $\Delta_M u \ge -c$,
    for
    some constant $c > 0$ where $\Delta_M$ is the
Laplace-Beltrami operator on
    $M$.
    Put $\mu(t) = |u^{-1}([t, +\infty))|_M$ and $t_0 = \sup_{p\in M} u(p)$. Then $\mu(t)$ is locally absolutely continuous and
    \begin{equation}\label{diff ineq}
      \mu'(t) \le -\,\frac{H(\mu(t))}{c\mu(t)}
    \end{equation}
    for almost all $t\in (-\infty, t_0)$.
  \end{theorem}
 Roughly speaking, this result tells us that $u$ cannot be too concentrated in the measure theoretic sense, which can be regarded as a new form of the uncertainty principle.

  The assumption that $u$ is a Morse function, unlike every other one, is
  purely
  technical for this theorem to hold. For a general function $u$ satisfying all
  the other conditions, we can get an almost equivalent result. To state it, it is convenient
to define, for $t_1 < t_2<t_0$ and $\mu>0$,
   $D(t_1, t_2, \mu):=g(t_1)$, where $g(t)$ is the solution, on the interval $[t_1,t_2]$, of the
   (backward)
  differential equation
  \begin{equation}\label{diffeq}
    g'(t) = -\,\frac{H(g(t))}{cg(t)}
  \end{equation}
with initial condition $g(t_2) = \mu$, provided that such solution exists.
  \begin{theorem}\label{main-mon}
    Let $M$ be an $n$-dimensional Riemannian manifold satisfying \eqref{isop} and let
    $u:M\to \R$ be a function in $C^2(M)$ such that for all $t\in \R$ the sets
    $u^{-1}([t, +\infty))$ are compact and $\Delta_M u \ge -c$ for some
constant
    $c > 0$, where $\Delta_M$ is the Laplace-Beltrami operator
on $M$. Put $\mu(t) = |u^{-1}([t, +\infty))|_M$ and $t_0 = \sup_{p\in M} u(p)$.
    Then for all $t_1 < t_2 < t_0$ we have
\begin{equation}
\label{claim2}
D(t_1, t_2, \mu(t_2))\le \mu(t_1)
\end{equation}
  \end{theorem}
 Part of the result is that the solution $D(t, t_2, \mu(t_2))$ exists for every $t<t_2$. This a consequence of the above a priori bound, which prevents blow-up in finite time in the (backward) Cauchy problem.
  If $u$ is a Morse function, then this theorem is a direct consequence of Theorem
  \ref{main-diff} and a basic comparison principle for first-order ODE. If the
  function
  $u$ is not Morse, then we can approximate it by Morse functions while
  preserving all the other assumptions and use the continuity of the solution
  to the
  differential equation on the initial conditions. For the reader's convenience we put the deduction of the Theorem \ref{main-mon} from the Theorem \ref{main-diff} in the Appendix.

  A version of this result was used in \cite{NicTil1} and \cite{Kulikov},
  the difference being that in these papers the authors worked with
weighted analytic functions of the form $f(z) =
g(z)e^{-\phi(z)}$, with $g$ holomorphic and $\phi$ having constant Laplacian,  for which we
have a lower bound of $\Delta \log |f(z)|$ instead of
  $\Delta |f(z)|$, which amount just to a simple change of variables. The
  advantages of Theorem~\ref{main-mon} are first of all that we can consider
  many
  more different manifolds than just Euclidean and hyperbolic spaces, in
  particular we can also work in the spherical geometry, but more generally on
  simply-connected two-dimensional manifolds of bounded curvature. Another
  advantage is that the functions that we work with are no longer analytic,
  thus
  vastly enlarging the domain of applicability of this theorem.

 As a consequence of the monotonicity result in Theorem \ref{main-mon} we will prove the following sharp functional inequality.
  \begin{theorem}\label{main-contr}
    Let $F:\R\to \R$ be a smooth increasing function 
    such that $\lim_{t\to-\infty}F(t)=0$
and 
   ${G:[0, F(t_0)]}\to\R$ be a continuous convex function with $G(0) = 0$, for some $t_0\in\R$.
Let $M$ be an
    $n$-dimensional Riemannian manifold satisfying \eqref{isop} and let $u:M\to
    \R$
    be a $C^2(M)$ function such that for all $t\in \R$ the sets 
$u^{-1}([t,
    +\infty))$ are compact and $\Delta_M u \ge -c$ for some constant $c > 0$,
    where
    $\Delta_M$ is the Laplace-Beltrami operator on $M$, with
    \begin{equation}\label{condition-contr}
      \int_M F(u(p))\, d{\rm Vol}(p) = 1.
    \end{equation}
    Let $\mu_0(t)> 0$ be a solution to the differential equation \eqref{diffeq}
    on
    $(-\infty, t_0)$, such that
    $$\int_{-\infty}^{t_0} F'(t) \mu_0(t)dt = 1$$
    and $\lim_{t\to t_0^-} \mu_0(t) = 0$.
    Then $u(p)\le t_0$ for all $p\in M$ and
    \begin{equation}\label{result-contr}
      \int_M G(F(u(p)))\, d{\rm Vol}(p)\le \int_{-\infty}^{t_0} G'(F(t))F'(t) \mu_0(t)dt.
    \end{equation}
    Moreover, if $G$ is not linear on $[0,F(t_0)]$, then equality in \eqref{result-contr}
    is possible only if either both integrals are $-\infty$ or $|u^{-1}([t,
    +\infty))|_M = \mu_0(t)$ for all $t<t_0$.

  \end{theorem}
  The function $F$ corresponds to the change of variables, for example $F(t) =
  e^t$ if we want to consider $\log$-subharmonic functions, while the function
  $G$
  corresponds to what we want to integrate, for example $G(t) = t^p$ if we want
  to
  consider $L^p$-norms.

\begin{remark}\label{remark zero}
Since we want to apply our theorem to logarithms of analytic functions, which 
can be $0$ at some points, sometimes it is convenient for us to assume that 
$u:M\to [-\infty, \infty)$ is continuous and $C^2$ on $u^{-1}(\R)$. This case 
follows from the above theorem since $u^{-1}(\R)$ is still a manifold without 
boundary, $u^{-1}([t, +\infty))$ are compactly embedded into it and $F(-\infty) 
=G(0)=0$ so the integrals do not change.
\end{remark}

By applying Theorem \ref{main-contr} to some particular instances of
manifolds $M$ and functions $u$ it is possible to prove that in many occasions
the most concentrated (normalized) functions in a reproducing kernel Hilbert 
space are given by the normalized reproducing kernels. This is
particularly clear when the space consists of holomorphic functions and there is
a group acting on $M$ which is compatible with the reproducing 
kernel structure. If
we quantify the concentration of the functions in terms of the Wehrl
entropy, this is essentially the content of Section~\ref{Wehrl_conjecture}.
We prove a generalized version of the conjecture, that gives as a corollary
the hypercontractive embeddings among spaces.

There is also a local analogous problem, where we inquire
which is the domain of a given measure where a normalized function in a
reproducing kernel Hilbert space is mostly concentrated. Again the extremal
functions for such problems are reproducing kernels and the extremal domains
are the corresponding super-level sets. This is the content of Section~\ref{local}.\bigskip

Rupert  Frank \cite{Frank} has independently and simultaneously
obtained analogous results to those in Sections~\ref{Wehrl_conjecture} and
\ref{local} in the three classical geometric models:
sphere, Euclidean plane and hyperbolic disk.  We have chosen to present a
streamlined proof in a more general version that covers as a particular case the
classical results. Moreover, our approach allows us to identify the maximizers
in the local estimates; see Section~\ref{local}.

  \section{Proof of Theorem \ref{main-diff}}
 Since $u$ is a Morse function, we have $|\{p\in M: \nabla u(p)=0\}|_M=0$, which implies that $\mu(t)$ is locally absolutely continuous and the coarea formula holds in the following form
  $$\mu'(t) = -\int_{\partial A_t} |\nabla u|^{-1}d\mathcal{H}^{n-1}$$
  for almost all $t\in\mathbb{R}$, where $A_t = u^{-1}([t,+\infty))$, $|\nabla
u|$ stands for the lenght of $\nabla u$ in the tangent space and
$\mathcal{H}^{n-1}=|\cdot|_m$ is the $n-1$-dimensional Hausdorff measure on $M$
(cf. \cite{Fed}*{3.2.12 and 3.2.46}). By Sard's theorem, for almost all
$t$ we have $\nabla u\not=0$ where $u=t$, so that $\partial
  A_t =u^{-1}(\{t\})$ is a smooth submanifold, which is compact by the assumption that $u^{-1}([t, +\infty))$ is compact for all $t\in\R$.

  Next, we apply the Cauchy--Schwarz inequality on $\partial   A_t$:
  $$|\partial A_t|_m^2 = \left(\int_{\partial A_t} d\mathcal{H}^{n-1}\right)^2 \le
  \int_{\partial A_t} |\nabla u|^{-1} \,d\mathcal{H}^{n-1}\int_{\partial A_t} |\nabla u|\,d\mathcal{H}^{n-1}.$$

  Now, $\nabla u$ is orthogonal to $\partial
  A_t$ and pointing inside $A_t$.
 Thus, denoting by $\nu$ the unit outward normal to
  $\partial A_t$, we have $|\nabla u| = -\nabla u \cdot \nu$. Plugging this in
  and using Gauss--Green's theorem, we have
  $$\int_{\partial A_t} |\nabla u| \,d\mathcal{H}^{n-1} = -\int_{\partial A_t} \nabla u \cdot \nu \,d\mathcal{H}^{n-1}=
  -\int_{A_t} \Delta u \,d\textrm{Vol}\le c|A_t|_M = c\mu(t).$$

  By the isoperimetric inequality, we have $|\partial A_t|_m^2 \ge H(\mu(t))$.
  Combining everything and dividing by $c\mu(t)$ (note that here we used that
$t
  < t_0$, that is $\mu(t) > 0$), we get
  $$-\mu'(t) \ge \frac{H(\mu(t))}{c\mu(t)}.$$
  Multiplying this by $-1$ we get the desired result.

\begin{remark}\label{remloc}
It turns out that the claim of Theorem \ref{main-diff} is of local nature,
and the assumption that every superlevel sets of $u$ is compact can be weakened.
In fact, the previous proof yields (without changes) the following more general result:
if the
superlevel sets $\{ u\geq t\}$ are compact for  $t> \tau$ (for some 
$\tau<t_0$), and $\Delta _M u \geq -c$
in the open set where $u>\tau$, then $\mu(t)$ is locally absolutely continuous in
$(\tau,t_0)$, and  \eqref{diff ineq} holds true for a.e. $t\in (\tau,t_0)$.
\end{remark}

  \section{Proof of the Theorem \ref{main-contr}}
 The following preliminary result will play a crucial role in the following.
  \begin{lemma}\label{lemma}
 With the same notation and assumptions as in Theorem~\ref{main-contr}, let 
$\mu(t) = |u^{-1}([t, +\infty))|_M$ and suppose $\mu\not=\mu_0$ at some point, 
where $\mu_0(t)$ in understood to be extended by $0$ past $t_0$.
There exist $ t_1 < t_0$ such that $\mu(t) \ge \mu_0(t)$
  if
  $t \le t_1$ and $\mu(t) < \mu_0(t)$ if $t_1 < t< t_0$. In
  particular,
  $\mu(t) = 0$ for $t\ge t_0$.
 Moreover $\mu(t) > \mu_0(t)$
  if
  $t_1-t>0$ is large enough.
  \end{lemma}
Note that this lemma already implies that $u(p)\le t_0$ for all $p\in M$.
  \begin{proof}
The condition
  \eqref{condition-contr} is equivalent to
  $$\int_{-\infty}^\infty F'(t)\mu(t)dt = 1.$$
  Hence
  $$\int_{-\infty}^\infty F'(t)\mu(t)dt = \int_{-\infty}^{\infty}
  F'(t)\mu_0(t)dt.$$
Since $F'(t)> 0$ and $\mu\not=\mu_0$, there should be
  $t_2$ and $t_3$ such that $\mu(t_2) > \mu_0(t_2)$ and $\mu(t_3)<
  \mu_0(t_3)$; in particular $t_3<t_0$, because $\mu_0(t)=0$ for $t\geq t_0$. By Theorem \ref{main-mon} for $t\le t_2$ we have $\mu(t) >
  \mu_0(t)$ while for $t_3\le t<t_0$ we have $\mu(t)< \mu_0(t)$.  We denote by $t_1$ the
  infimum of
  admissible $t_3$'s. The conclusion is then clear.
    \end{proof}

\begin{proof}[Proof of Theorem \ref{main-contr}]
As in Lemma \ref{lemma} we set $\mu(t) = |u^{-1}([t, +\infty))|_M$ for $t\in\R$.

The left-hand side of \eqref{result-contr} is equal to
  $$\int_{-\infty}^{t_0} G'(F(t))F'(t)\mu(t)dt,$$
 because $\mu(t)=0$ for $t\geq t_0$ by Lemma \ref{lemma}.

Observe that this latter integral could be $-\infty$ but not $+\infty$. Indeed, 
the positive part of $G'(F(t))$ is bounded outside of the vicinity of $t_0$ and 
near $t_0$ we have that $\mu(t)$ is bounded while 
$\int_{-\infty}^{t_0} G'(F(t))F'(t)\, dt = G(F(t_0))<\infty$ is 
convergent. The same can be said for the right-hand side of 
\eqref{result-contr}.  Moreover it is clear that we have an equality in 
\eqref{result-contr} if $\mu(t)=\mu_0(t)$ for $t<t_0$.

Now, if the right-hand side of \eqref{result-contr} is $-\infty$ then $G'(x)<0$ for $x>0$ small enough and
$$\int_{-\infty}^{\overline{t}} G'(F(t))F'(t)\mu_0(t)dt=-\infty,$$
if $t_0-\overline{t}>0$ is large enough, because $G'$ is bounded below on the compact subintervals of $(0,F(t_0)]$. By Lemma \ref{lemma} the same holds with $\mu_0$ replaced by $\mu$, which implies that the left-hand side
of \eqref{result-contr} is $-\infty$ as well.

Suppose now that both sides of \eqref{result-contr} are finite. Let $t_1$ be as in Lemma \ref{lemma}. We must have
  \begin{align*}\int_{-\infty}^{t_0} G'(F(t))F'(t)(\mu_0(t)-\mu(t))dt =
    \int_{-\infty}^{t_1}(G'(F(t))-G'(F(t_1)))F'(t)(\mu_0(t)-\mu(t))dt \\
    +
    \int_{t_1}^{t_0}(G'(F(t))-G'(F(t_1)))F'(t)(\mu_0(t)-\mu(t))dt \ge 0
  \end{align*}
  where in the last step we used Lemma \ref{lemma} and that $G'$ is non-decreasing. To be precise, $G'$ is in fact defined  only almost everywhere, but the above formulas hold true if  $G'$ is understood e.g. as the left derivative, so that it is an everywhere defined non-decreasing function on $(0,+\infty)$ (and the pointwise value $G'(F(t_1))$ makes sense).

   If we have equality in the latter estimate, we have $$(G'(F(t))-G'(F(t_1)))(\mu_0(t)-\mu(t))=0$$ for almost every $t<t_1$ and for almost every $t\in (t_1,t_0)$, hence for almost every $t<t_0$. Since $\mu(t)<\mu_0(t)$ if $t\in (t_1,t_0)$ we have $G'(F(t))=G'(F(t_1))$ for almost every $t\in (t_1,t_0)$. On the other hand by Lemma \ref{lemma} $\mu_0(t)<\mu(t)$ for $t<\overline{t}$, for some $\overline{t}<t_0$. Hence $G'(F(t))=G'(F(t_1))$ for almost every $t<\overline{t}$. Since $G'(F(t))-G'(F(t_1))$ is non-positive and non-decreasing for $t<t_1$ we deduce that $G'(F(t))=G'(F(t_1))$ for every $t<t_1$. Summing up, $G'(F(t))=G'(F(t_1))$ for almost every $t<t_0$ (in fact, for every $t<t_0$) and therefore $G(t)$ is affine on $(F(-\infty), F(t_0))=(0,F(t_0))$, and therefore linear on $[0,F(t_0)]$, because $G$ is continuous and $G(0)=0$.
 \end{proof}

  \section{Applications, the generalized Wehrl
conjecture}\label{Wehrl_conjecture}

  In \cite{Wehrl}, Wehrl conjectured that among all Glauber states, the
  coherent states minimize the Wehrl entropy. To be more precise, we recall the
  basic terminology.

  We are given a locally compact group $G$ and a unitary representation $T$ of
  $G$ on a Hilbert space $\mathcal H$. We fix a vector $\psi \in \mathcal H$
  and denote by $H\subset G$ the subgroup that leaves $\psi$ invariant by the
  action $T$ on elements of $H$ up to a unimodular factor, i.e., $T[h](\psi) =
  e^{i\theta_h} \psi$ for all $h\in H$.  Let $X = G/H$. In many instances, the
  Haar measure on $G$ induces a measure $\mu$ on $X$ that is invariant under
  the action of $G$. Then for any coset $x\in X$
  we take a representative $g(x)$ and define the state
  $v_x = T[g(x)] (\psi)$. This vector is well-defined up to a unimodular
  factor that may change with the representative $g(x)$ that has been chosen.
  These vectors are the coherent states. For every $u\in \mathcal H$ and every
  $x\in X$ we may form $u(x) = \langle u, v_x\rangle$. In this way, we may
think
  of $\mathcal H$ as a reproducing kernel Hilbert space of functions over $X$.
  We have that
  $\langle v_x,v_y\rangle = K(x,y)$ is the reproducing kernel for $\mathcal H$,
  i.e.,   for all $v\in \mathcal H$, $v(x) = \int_X K(x,y) v(y)\, d\mu(y)$.

  Given a vector $v\in\mathcal H$ of norm one we define the Wehrl entropy as
  \[
  \int_{X} -|v(x)|^2 \log |v(x)|^2 \, d\mu(x).
  \]
  The conjecture is that this is minimized for the coherent states. Namely,
  the conjecture postulates that among the functions with unit norm the
  reproducing kernels are the most concentrated. Sometimes a more general
  conjecture is formulated, replacing the function $f(x)= x\log(1/x)$ in the
  Wehrl entropy definition by any other concave function $f:[0,1]\to \mathbb R$.

  This  conjecture was originally formulated by Wehrl  in \cite{Wehrl} for Glauber
  states. In this   original setting $\mathcal H$ is the Fock space of entire
  functions such that   $\int_{\mathbb C} |f(z)|^2 e^{-|z|^2} dA(z) <\infty$
  and $X = \mathbb C$ and $G$ is the Heisenberg group. It   was proved by Lieb
  that the coherent states are minimizers in \cite{Lieb}. Later on, Carlen in
  \cite{Carlen} found a new proof that moreover confirmed that these are the
  unique minimizers.

  In \cite{Lieb} Lieb extended the conjecture for the Bloch states. In this
  case, the group is $SU(2)$ and $\mathcal H$ is a space of holomorphic
  polynomials of degree up to $j$ endowed with the Fubini-Study metric. The
  coherent states associated are the corresponding reproducing kernels, see
  Subsection~\ref{Bloch} for details. Thirty-six years later, in
  \cite{LiebSolo1} Lieb and Solovej proved that the reproducing kernels are
  minimizers for the Wehrl entropy. The fact that these are the only minimizers
  remained open. They expect that a similar result should hold for any
  semi-simple Lie group.

  In \cite{LiebSolo3} they formulated the analogous problem for the group
  $SU(1,1)$ where the reproducing kernel Hilbert space is the Bergman space and
  $X$ is the unit disk and proved some partial cases. The full conjecture in
  this case was proved in \cite{Kulikov}. We will see now how all these cases
  and possibly many other instances of the Wehrl conjecture follow from our
  scheme. As a bonus, we will prove the uniqueness of the minimizers, thus setting
  the last piece of the Lieb conjecture for $SU(2)$.

  \subsection{Bloch coherent states}\label{Bloch}
  In \cite{LiebSolo1} Lieb and Solovej proved the generalized Wehrl conjecture
  that
  states that the Wehrl entropy is minimized at the coherent states in the
  Hilbert spaces of the irreducible representations of $SU(2)$. They did not
  prove that the coherent states alone minimize the entropy. In
  \cite{LiebSolo2} they
  extended their results to symmetric $SU(N)$ coherent states.

  To define the space of functions that we will consider, we first introduce
  the
  spherical measure on $\Cm$, which corresponds to
  the metric inherited from the Euclidean metric restricted to the sphere of
  radius $\frac{1}{2\sqrt{\pi}}$ transported to $\Cm$ by the stereographical
  projection. Namely, on $\Cm\ni z = x + iy$ we consider the Riemannian metric $\pi^{-1}(1+|z|^2)^{-2}(dx^2+dy^2)$ and the corresponding measure
  $$dm(z) = \frac{1}{(1+|z|^2)^2} \frac{dxdy}{\pi}.$$
  We will also sometimes denote the spherical measure of the set $A$ by $|A|_M=
  m(A)$.
  \begin{definition} Let $j\in \mathbb N $. We define
    $\mathcal P_j$ as the finite
    dimensional space of polynomials:
    \[
    z \mapsto \sum_{k= 0}^{j} c_k z^k,
    \]
    with inner product:
    \[
    \langle f, g\rangle = (j+1) \int_{\mathbb C} \frac{f(z) \overline
      g(z)}{(1+|z|^2)^j}
    dm(z)
    \]
    and reproducing kernel
    \[
    K_j(z,w) =  (1+ z\overline w)^{j}.
    \]
    The functions $K_j(\cdot,w)$ are the coherent states.
  \end{definition}

For each $p> 0$  the
$p$-(quasi-)norm of $f\in \mathcal P_j$ is defined as
  \[
  \|f\|_{\mathcal P_j,p}^p:= (pj/2+1) \int_{\mathbb C}
  \left|\frac{f(z)}{(1+|z|^2)^{j/2}}\right|^p\, dm(z).
  \]
  The factor $(pj/2+1)$ is introduced to guarantee that $\|1\|_{\mathcal P_j,
    p}
  = 1$. Thus, the subharmonicity of $|f|^p$ and integration in polar coordinates yield $|f(0)| \le \|f\|_{\mathcal P_j,
    p}$.
  There is an invariance of the space $\mathcal P_j$ under a subgroup of the
  Möbius transformations
  that preserve the spherical metric, i.e., for any $\alpha, \beta \in \Cm$
  such that $|\alpha|^2+|\beta|^2 = 1$ the map
  $T_{\alpha, \beta } f = ( \beta z + \bar\alpha )^j f\left(\frac{\alpha z
    -\bar \beta}{\beta z + \bar\alpha }\right)$ is an isometry of $\mathcal
  P_j$
  for all $p$. This is, in fact, the unitary representation of $SU(2)$ on
  $\mathcal P_j$.

  This entails that for all $f\in \mathcal P_j$ and all $p> 0$ we have:
  \begin{equation}\label{pointwise}
    \sup_{z\in \Cm} \frac{|f(z)|}{(1+|z|^2)^{j/2}} \le
    \|f\|_{\mathcal P_j,p}.
  \end{equation}

  Then the Conjecture 3.5 of Bodmann in \cite{Bodmann} is that
  \begin{equation}\label{contractive}
    \|f\|_{\mathcal
      P_j,q} \le \|f\|_{\mathcal P_j,p}
  \end{equation}
  when $1\le p \le q$ and equality is achieved if and only if $f$ is a multiple
  of
  the
  reproducing kernel.
  The generalized Wehrl theorem in this context is that for any $f\in \mathcal
  P_j$ normalized such that $\|f\|_{\mathcal P_j,2} = 1$, it holds:
  \[
  S_j(|f|^2) := -(j+1) \int_{\mathbb C} \frac{|f(z)|^2}{(1+|z|^2)^j}\ln
  \frac{|f(z)|^2}{(1+|z|^2)^j} dm(z) \ge \frac{j}{j+1}.
  \]
  This inequality follows from \eqref{contractive} observing
  that $\frac{\partial \|f\|_p}{\partial p}(2) \le 0$ or directly from the theorem below.

  Our aim is to prove:
  \begin{theorem}\label{Bergman}
    Let $G:[0,1]\to \R$ be a continuous convex function such that $G(0) =
    0$, $j\in\mathbb{N}$, $p> 0$.
    Then the maximum value of
    \begin{equation}\label{bergmanineq}
      \intt_\Cm G\left(\frac{|f(z)|^p}{(1+|z|^2)^{pj/2}}\right)dm(z)
    \end{equation}
   subject to the condition that $f\in
    \mathcal P_j$ and $||f||_{\mathcal P_j,p}=1$,  is attained for $f(z) = ( \beta z + \bar\alpha )^j$, for any $\alpha,\beta\in\mathbb{C}$, $|\alpha|^2+|\beta|^2=1$. If $G$ is not linear on $[0,1]$, then these are the only maximizers.
  \end{theorem}

The first part of the statement corresponds to Theorem~2.1 in \cite{LiebSolo1}. The uniqueness of the maximizers is new. We observe that the expression in \eqref{bergmanineq} is always finite, since $m(\mathbb{C})=1$ is finite.

  \begin{proof}
    Take $M$ to be the sphere in $\mathbb R^3$ of radius
    $\frac{1}{2\sqrt{\pi}}$ with the Riemannian metric
    inherited from $\mathbb R^3$.
    Paul Levy's isoperimetric inequality for the sphere (see e.g. \cite{Oss}) says that for any open set
    $A\subset M$ with smooth boundary we have
    $$|\partial A|_{\mathcal{H}^1}^2 \ge 4\pi |A|_M - 4\pi|A|_M^2.$$
    Thus, we have \eqref{isop} with $H(x) = 4\pi x(1-x)$.
   For any polynomial $f\in \mathcal P_j$, take $u = \log\left(
    \frac{|f(z)|}{(1+|z|^2)^{j/2}}\right)$, pull it back to the sphere via the stereographic projection and extend to the North Pole $N$ by continuity $u(N) = \log(|c_j|)$, where $c_j$ is the coefficient of $z^j$ in $f(z)$. We will apply Theorem~\ref{main-contr} (and Remark \ref{remark zero}) to the function $u$ and $M$.

    The spherical Laplacian in
    stereographic
    coordinates is $\Delta_M = \pi (1+|z|^2)^2 \Delta_e$, where $\Delta_e$ is
    the
    ordinary Euclidean Laplacian, thus
    $$\Delta_M u = \pi (1+|z|^2)^2 \Delta_e \log (|f|) -  \pi (1+|z|^2)^2
    \Delta_e
    \log (1+|z|^2)^{j/2}\ge -2\pi j $$ and $F(t) = (pj/2+1)e^{pt}$.  We assume
    that the
    polynomial is normalized, i.e.,
    $$
    1=\|f\|_{\mathcal P_j, p}^p = \int_{\mathbb C}
    F(u(z)) dm(z).$$
    In order to apply Theorem~\ref{main-contr}, we  identify $\mu_0$. The
    function
    $\mu_0(t)$  is the solution to
    $$ g'(t) = 4\pi\frac{g(t) -g^2(t)}{-\pi 2jg(t)} = \frac{g(t)-1}{j/2},
    $$
    with the normalization $$(pj/2+1)\int_{-\infty}^{t_0} p e^{p t} \mu_0
    (t)\, dt = 1$$ and $\lim_{t\to t_0^-}\mu_0(t) = 0$.
    The solution is attained when $t_0= 0$ and  $\mu_0(t) = 1 - e^{2t/j}$ when
    $t\in
    (-\infty, 0)$. This is exactly $$m\Big(\Big\{z\in\mathbb{C}: \log \frac{1}{(1+|z|^2)^{j/2}} > t\Big\}\Big),$$
    thus
    $f = 1$ attains the maximum. Any other coherent state $T_{\alpha, \beta} 1
    = (\beta z + \bar \alpha)^j$,
    with $|\alpha|^2+|\beta|^2 = 1$ has the same distribution function, thus it
    will also attain the maximum. Let us check that they are the only
    maximizers. Indeed, if $f\in \mathcal P_j$ is a maximizer with
    $\|f\|_{\mathcal P_j, p} = 1$ we may assume (after an application of
    $T_{\alpha,\beta}$) that $\sup_{z\in \mathbb C}
    \frac{|f(z)|}{(1+|z|^2)^{j/2}}$ is attained at $z = 0$ (note that it must be attained somewhere since the sphere is compact). Subharmonicity
     implies that $|f(0)| \le 1$ and the equality is attained only when
    $f$ is constant, which can be seen by integrating in polar coordinates and applying subharmonicity to the function $|f(z)|^p$ on each circle $\{|z|=r\}$ and noting that subharmonicity is strict for large enough $r$ unless $f$ is constant. On the other
    hand by Theorem~\ref{main-contr} if $f$ is a maximizer then $|u^{-1}([t,+\infty))|_M
    = \mu_0(t)> 0$ if $t < 0$. Thus $\sup u = 0$, i.e., $\sup_{z\in\mathbb{C}}
    \frac{|f(z)|}{(1+|z|^2)^{j/2}} = 1$.
  \end{proof}
  \subsection{Glauber coherent states}
For $p>0$, $\alpha>0$, we consider the Bargmann-Fock space of
  entire functions $f$ of one complex variable $z=x+iy$ with $p$-(quasi-)norm
  \[
  \frac{p\alpha}{\pi}\int_{\mathbb C} |f(z) e^{-\alpha |z|^2}|^p dA(z) < \infty,
  \]
 with $dA(z)=dx dy$. For $p=2$ we have a reproducing kernel Hilbert space and the coherent states are given by $e^{ \alpha  \bar a z -\alpha |a|^2/2}$, with $a\in\mathbb{C}$, see e.g. \cite{Zhu}.

  The
  result that follows from Theorem~\ref{main-contr} is:
  \begin{theorem}
    Let $G:[0,1]\to \R$ be a convex function such that
    $G(0) = 0$. Let $\alpha>0$, $p>0$.
    Then the supremum of the functional
    \begin{equation}
      \int_\Cm G\left(|f(z)e^{-\alpha |z|^2/2}|^p\right)dA(z)
    \end{equation}
    subject to the condition
    that $f\in
    \mathcal H(\mathbb C)$ and
    \begin{equation}\label{constraint}
    \frac{p\alpha}{\pi}\int_{\mathbb C} |f(z) e^{-\alpha |z|^2/2}|^p dA(z)= 1
    \end{equation}
    is attained at $f(z) = e^{\alpha \bar a z-\alpha |a|^2/2}$ for any $a\in\mathbb{C}$.
    If $G$ is not linear on $[0,1]$, and this supremum is finite (i.e. $>-\infty$), then these are the only maximizers, up to a unimodular factor.
  \end{theorem}

We emphasize that the above functional takes values in $[-\infty,+\infty)$ and its supremum can be finite or $-\infty$, depending on $G$.
  This theorem for a general convex function $G$ was proved by Lieb and Solovej
  in \cite{LiebSolo1}. The fact that the minimizers are unique for a general
  convex function is new, as far as we know. For the classical Wehrt entropy, the uniqueness
  was
  proved by Carlen in \cite{Carlen}.
  \begin{proof}
    The operators $$T_a f(z) = e^{\alpha \bar a z -\alpha |a|^2/2}f(z-a),$$ with $a\in\mathbb{C}$, are an isometry in 
    the
    Fock spaces. Moreover, under the assumption \eqref{constraint},  $|f(z) e^{-\alpha |z|^2/2}|\le 1$ for all
$z\in\mathbb
    C$ (see e.g. \cite{Zhu}).
    The result follows by applying Theorem~\ref{main-contr} using the classical
    isoperimetric inequality in the plane, that is \eqref{isop} with $H(x) = 4\pi x$ (see e.g. \cite{Oss}), and taking $u =
    \log(|f(z)|e^{-\alpha |z|^2/2})$, hence $\Delta u=-2\alpha$, and 
$F(t) = \frac{p\alpha}{2\pi} e^{p t}$. Here $\mu_0(t)=-2\pi 
t/\alpha$ for $t\in (-\infty,0)$; hence $t_0=0$. The uniqueness
    of
    the maximizers follows as in Subsection~\ref{Bloch} (here the supremum $\sup_{z\in\mathbb{C}} u$ is attained because $\lim_{z\to\infty} |f(z)|e^{-\alpha |z|^2/2}=0$, cf. \cite{Zhu}).
  \end{proof}
  \subsection{SU(1,1) coherent states}
  Now we consider, for $\alpha>0$, $p>0$, the weighted Bergman space of analytic functions $f$ in
  the
  unit disk $\mathbb D\subset \mathbb C$, with $p$-(quasi-)norm
  \[
  \int_{\mathbb D} (\alpha-1) |f(z)|^p (1-|z|^2)^{\alpha}\,dm(z)<\infty,
  \]
 where
 $$dm(z)=\frac{dxdy}{\pi(1-|z|^2)^2}$$ is the area element, for $z=x+iy$.
  When $p=2$ we obtain a reproducing kernel Hilbert space, and the coherent states are given by $(1-z\bar
      a)^{-2\alpha/p}$, $a\in\mathbb{D}$; see \cite{Hed_etal}.

      The issue addressed in the previous subsections was reformulated for these spaces as a function
  theory problem in \cite{LiebSolo3}. This problem had been
  considered, and some partial solutions found in \cite{Brevig} and
  \cite{Bayart}. Finally,  the following theorem was proved in
\cite{Kulikov}*{Theorem 1.2 and Remark 4.3}.
    \begin{theorem}
    Let $G:[0,1]\to \R$ be a continuous convex function such that
    $G(0) = 0$.
    Let $\alpha > 1$, $p>0$. The supremum of the functional
    \begin{equation}
      \int_{\mathbb D} G\left(|f(z)|^p(1-|z|^2)^\alpha\right)dm(z)
    \end{equation}
    subject to the condition
    that $f\in
    \mathcal H(\mathbb D)$ and
    $$\int_{\mathbb D} (\alpha-1) |f(z)|^p (1-|z|^2)^{\alpha}\,dm(z) = 1$$
    is attained at $f(z) = \frac{(1-|a|^2)^{\alpha/p}}{(1-z\bar
      a)^{2\alpha/p}}$ for any $a\in \mathbb D$.
    If $G$ is not linear on $[0,1]$, and this supremum is finite (i.e. $>-\infty$), then these are the only maximizers, up to a unimodular factor.
  \end{theorem}
  \begin{proof}
    Again this is now a corollary of Theorem~\ref{main-contr} where the
    manifold is $\mathbb D$ endowed with the hyperbolic metric, the function
    $u(z) = \log (|f(z)|(1-|z|^2)^{\alpha/p})$, the function $F(t) =
    (\alpha-1)\exp(pt)$ and we have the isoperimetric inequality \eqref{isop} in the
    hyperbolic space with $H(x) = 4\pi (x+x^2)$ (see e.g. \cite{Oss}). The Laplace-Beltrami operator now is given by $\Delta_{\mathbb{D}}=\pi(1+|z|^2)^2\Delta_e$, where $\Delta_e$ is the Euclidean Laplace operator. A straitforward computation shows that
    \[
  \Delta_{\mathbb{D}} u\geq -4\pi\alpha/p,
   \]
   which yields $\mu_0(t)= e^{-pt/\alpha}-1$ for $t\in(-\infty,0)$; hence $t_0=0$.  The uniqueness
    of
    the maximizers follows as in Subsection~\ref{Bloch}.
  \end{proof}

  \section{Local estimates: Faber-Krahn inequalities}\label{local}
 In this section we prove a local counterpart of the estimate in 
Theorem \ref{main-contr}. A similar result had first appeared in the special 
case of Glauber coherent states in \cite{NicTil1}. The following theorem 
provides a far reaching generalization of that result, and the proof is even 
simpler.
 \begin{theorem}\label{thm faber-krahn}
 Under the same assumption and notation of Theorem \ref{main-contr}, suppose in addition that $G(x)>0$ for $x>0$.  Then for every set $\Omega\subset M$ and every $u$ as in Theorem \ref{main-contr} we have
\begin{equation}\label{result-contr2}
     \int_\Omega G(F(u(p)))\, d{\rm Vol}(p) \le \int_{0}^{|\Omega|_M} G(F(\mu_0^{-1}(s)))ds.
    \end{equation}
    Moreover equality in \eqref{result-contr2}
    is possible for some $u$ as above and $\Omega$ with $|\Omega|_M>0$ if and only if $|u^{-1}([t,
    +\infty))|_M = \mu_0(t)$ for all $t<t_0$ and $\Omega=u^{-1}([t,
    +\infty))$, with $t=\mu_0^{-1}(|\Omega|_M)$ (up to null sets) if $|\Omega|_M<|M|_M$, or $\Omega=M$ if $|\Omega|_M=|M|_M$.
 \end{theorem}
 Observe that the additional assumption $G(x)>0$ for $x>0$ implies that  $G:[0,+\infty)\to\mathbb{R}$ is strictly increasing. Also, we have the characterization of the maximizers without any further assumption on $G$ (such as non-linearity).
 \begin{proof}
For $u$ as in the statement, let $\mu(t) =\mu_u(t)= |u^{-1}([t, +\infty))|_M$, $t\in\mathbb{R}$, be its distribution function and $u^\ast(s)=\sup\{t:\mu(t)<s\}$, for $0\le s< |M|_M$, its non-increasing rearrangement.  If $|\Omega|_M<|M|_M$, let $\tilde{\Omega}\subset M$ be any subset with $|\tilde{\Omega}|_M=|\Omega|_M$ and $u^{-1}((t,+\infty))\subset \tilde{\Omega}\subset u^{-1}([t,+\infty))$, with $t=u^\ast(|\Omega|_M)$ (up to null sets). If $|\Omega|_M=|M|_M$, let $\tilde{\Omega}=M$.  Then it is easy to check that
 \begin{equation}\label{eq 1}
  \int_\Omega G(F(u(p)))\, d\textrm{Vol}(p)\leq  \int_{\tilde{\Omega}} G(F(u(p)))\, d\textrm{Vol}(p)= \int_{0}^{|\Omega|_M} G(F(u^\ast(s)))ds
 \end{equation}
 where the equality follows from the fact that $u$ and $u^\ast$ are equi-measurable and the Fubini theorem. Hence we are going to prove that
 \begin{equation}\label{eq 2}
 \int_{0}^{s} G(F(u^\ast(
 \tau)))d\tau\leq \int_{0}^{s} G(F(\mu_0^{-1}(\tau)))d\tau
 \end{equation}
for $0\leq s\leq |M|_M$. This is clear if $\mu(t)=\mu_0(t)$ for $t<t_0$. Suppose then that $\mu\not=\mu_0$ at some point.
Consider the function
\[
\varphi(s):= \int_{0}^{s} G(F(\mu_0^{-1}(\tau)))d\tau - \int_{0}^{s} G(F(u^\ast(
 \tau)))d\tau
\]
for $0\leq s<|M|_M$. Clearly $\varphi$ is continuous and $\varphi(0)=0$. As a consequence of Lemma \ref{lemma} $\varphi$ is strictly increasing on $[0,\mu_0(t_1)]$; indeed, for $t_1<t<t_0$ we have $\mu(t)<\mu_0(t)$ which implies that $\mu_0^{-1}(s)>u^\ast(s)$ for $0<s<\mu_0(t_1)$. Similarly, on $[\mu_0(t_1),|M|_M)$ $\varphi$ is non-increasing (in fact strictly decreasing for $s<|M|_M$ large enough). Finally
\begin{align*}
\lim_{s\to|M|_M^-} \varphi(s)&= \int_{0}^{|M|_M} G(F(\mu_0^{-1}(\tau)))d\tau - \int_{0}^{|M|_M} G(F(u^\ast(
 \tau)))ds\\
 &= \int_0^{t_0} G'(F(t))F'(t)\mu_0(t)dt - \int_M G(F(u(p)))\, d\textrm{Vol}(p)>0
\end{align*}
where the latter inequality follows from Theorem \ref{main-contr}. As a consequence, $\varphi(s)>0$ for $0< s<|M|_M$. This concludes the proof of \eqref{result-contr2}.

The characterization of the cases of equality, also follows from the above discussion. The claim about $\Omega$ follows from \eqref{eq 1}, since in that case equally occurs in \eqref{eq 1} and the level sets  $u^{-1}(\{t\})$ have zero measure, because $\mu_0(t)$ is continuous (if $|\Omega|_M=|M|_M$, hence $\tilde{\Omega}=\Omega$, we also use the fact that $G(F(u(p)))$ is continuous and strictly positive on $M$).
 \end{proof}
 We now specialize the above result to the three geometries (spherical, Euclidean, hyperbolic). We begin with the spherical case and we use the notation of Subsection~\ref{Bloch}; in particular, for $z = x + iy\in \Cm$,  $dm(z) = \frac{1}{(1+|z|^2)^2} \frac{dxdy}{\pi}.$ Observe that $m(\mathbb{C})=1$.
  \begin{theorem}\label{Bergman fk}
    Let $G:[0,1]\to \R$ be a continuous convex function such that $G(0) =
    0$ and $G(x)>0$ for $x>0$. Let $j\in\mathbb{N}$, $ p> 0$.
Then for every $f\in
    \mathcal P_j$ with $||f||_{\mathcal P_j,p}=1$ and $\Omega\subset\mathbb{C}$,

      \begin{equation}\label{bergmanineq fk}
      \intt_\Omega G\left(\frac{|f(z)|^p}{(1+|z|^2)^{pj/2}}\right)dm(z)\leq \int_0^{m(\Omega)} G((1-s)^{pj/2})ds.
    \end{equation}
  Equality occurs in \eqref{bergmanineq fk} for some $f$ and $\Omega$ with $m(\Omega)>0$ if and only if
 $f(z) = ( \beta z + \bar\alpha )^j$ for some $\alpha,\beta\in\mathbb{C}$, 
$|\alpha|^2+|\beta|^2=1$ and  $\Omega$ is (up to a null set) a 
superlevel set of $|f(z)|/(1+|z|^2)^{j/2}$ (which is a disk in $\mathbb{C}$) if 
$m(\Omega)<1$, or $\Omega=\mathbb{C}$ if $m(\Omega)=1$.
  \end{theorem}
  \begin{proof}
We apply Theorem \ref{thm faber-krahn}, by arguing as in the proof of Theorem  \ref{Bergman}. In particular we have $\mu_0(t)=1-e^{2t/j}$ for $t\in (-\infty,0)$, hence if $u(z):=\log(|f(z)|/(1+|z|^2)^{j/2})$ is an extremal function, its maximum value is $t_0=0$, and the maximum of $|f(z)|/(1+|z|^2)^{j/2}$ is $1$. This gives the desired extremal functions.
  \end{proof}
Similarly we obtain the following results for Glauber and $SU(1,1)$ coherent states, that we state without proof.

 The  following result generalizes \cite{NicTil1}*{Theorem 3.1}, which
corresponds to the special case $p=2, \alpha=\pi$ and $G(x)=x$. Here
$\mathbb{C}$ is endowed with the Lebesgue measure $dA(z)=dxdy$, with $z=x+iy$.
We also write $|\Omega|=A(\Omega)$.

 \begin{theorem}\label{Bargmann fk}
    Let $G:[0,1]\to \R$ be a continuous convex function such that $G(0) =
    0$ and $G(x)>0$ for $x>0$. Let $\alpha>0$, $p>0$. Then for every $f\in \mathcal{H}(\mathbb{C})$ satisfying
$$
    \frac{p\alpha}{\pi}\int_{\mathbb C} |f(z) e^{-\alpha |z|^2/2}|^p dA(z)= 1
$$
and $\Omega\subset\mathbb{C}$, we have
      \begin{equation}\label{bargmannineq fk}
      \intt_\Omega G\left(|f(z)e^{-\alpha |z|^2/2}|^p\right)dA(z)\leq \int_0^{|\Omega|} G(e^{-p\alpha s/(2\pi)})ds.
    \end{equation}
  Equality occurs in \eqref{bargmannineq fk} for some $f$ and $\Omega$ with $|\Omega|>0$ if and only if
 $f(z) = e^{\alpha \bar a z-\alpha |a|^2/2}$, for some $a\in\mathbb{C}$, up a 
unimodular factor, and  $\Omega$ is (up to a null set) a superlevel 
set of $|f(z)|e^{-\alpha |z|^2/2}$ (which is a disk in $\mathbb{C}$) if 
$|\Omega|<\infty$, or $\Omega=\mathbb{C}$ if $|\Omega|=\infty$.
  \end{theorem}
  Finally, the following result generalizes \cite{RamTil}*{Theorem 3.1}, which corresponds to the particular case $p=2$, $G(x)=x$. Here the
unit disk $\mathbb{D}$ is endowed with the measure $dm(z) =
\frac{1}{(1-|z|^2)^2} \frac{dxdy}{\pi},$ for $z=x+iy\in  \mathbb{D}$.
 \begin{theorem}\label{Bergmanhyp fk}
    Let $G:[0,1]\to \R$ be a continuous convex function such that $G(0) =
    0$ and $G(x)>0$ for $x>0$. Let $\alpha>1$, $p>0$. Then for every $f\in\mathcal{H}(\mathbb{D})$ satisfying
     $$\int_{\mathbb D} (\alpha-1) |f(z)|^p (1-|z|^2)^{\alpha}\,dm(z) = 1$$
and $\Omega\subset\mathbb{C}$, we have
      \begin{equation}\label{bergmanhypineq fk}
      \intt_{\mathbb D}  G\left(|f(z)|^p(1-|z|^2)^\alpha\right)dm(z)\leq \int_0^{m(\Omega)} G((1+s)^{-\alpha})ds.
    \end{equation}
  Equality occurs in \eqref{bergmanhypineq fk} for some $f$ and $\Omega$ with $m(\Omega)>0$ if and only if
 $f(z) = \frac{(1-|a|^2)^{\alpha/p}}{(1-z\bar
      a)^{2\alpha/p}}$ for some $a\in \mathbb D$, up to a unimodular factor, and
 $\Omega$ is (up to a null set) a superlevel set of 
$|f(z)|(1+|z|^2)^{\alpha/p}$
(which is a disk in $\mathbb{D}$) if $m(\Omega)<\infty$, or
$\Omega=\mathbb{D}$ if $m(\Omega)=\infty$.
  \end{theorem}

  \section{Appendix, Proof of Theorem \ref{main-mon}}

 First we establish the Theorem in the case when $u$ is a Morse function (we may assume that $u$ is not constant,
 so that $\mu(t)>0$ for every $t<t_0$).
 In this case, by Theorem \ref{main-diff},
 $u$ is locally absolutely continuous on $(-\infty,t_0)$ and satisfies the differential inequality
 \eqref{diff ineq}
for a.e. $t<t_0$.
 Now consider an arbitrary $t_2<t_0$, and let $g(t)$ be the solution of the (backward) Cauchy problem
 \begin{equation}
 \label{cauchy}
 g(t_2)=\mu(t_2),\qquad g'(t)=-\,\frac {H(g(t))}{c g(t)},\quad t\leq t_2,
 \end{equation}
whose existence on some interval $(a,t_2]$ is guaranteed by the smoothness of $H>0$ and the fact that $\mu(t_2)>0$.
On this interval, combining \eqref{diff ineq} and \eqref{cauchy}, by a standard comparison theorem for ODEs (see e.g. Chapter 1 in \cite{BirkRota}) we obtain that $g(t)\leq \mu(t)$ and, since $\mu$ is locally bounded due to the
assumption that the level sets $\{u\geq t\}$ are compact, this prevents blow up in finite time for $g(t)$.
Therefore, the existence of the solution $g(t)$ (together with the bound $g(t)\leq \mu(t)$) propagates
 $t\leq a$,  and eventually one obtains \eqref{claim2} as claimed.

Now let $u$ be as in Theorem \ref{main-mon}. Since Morse functions are dense in 
the strong $C^2$ topology on $M$ (see e.g. 
\cite{Hirsch}*{Chapter~6, Theorem~1.2}), we can pick a sequence $u_n$ of Morse 
functions such that $|u-u_n|< \frac{1}{n}$, $|\Delta u - \Delta u_n|< 
\frac{1}{n}$.

Let $A_n(t) = u_n^{-1}([t, +\infty))$, $\mu_n(t) = |A_n(t)|_M$ and $A(t) = u^{-1}([t, +\infty))$. Fix numbers $t_1 < t_2 < t_0$. Note that for all $t < t_0 - \frac{1}{n}$ we have $A_n(t+\frac{1}{n}) \subset A(t)\subset A_n(t-\frac{1}{n})$. Applying Morse version of the theorem to the function $u_n$ we get for big enough $n$
$$D_n\left(t_1 + \frac{1}{n}, t_2 - \frac{1}{n}, \mu_n\left(t_2 - \frac{1}{n}\right)\right) \le \mu_n\left(t_1 + \frac{1}{n}\right),$$
where $D_n(t_3, t_4, \mu)$ is the solution to the differential equation $$g'(t) = -\frac{H(g(t))}{(c+\frac{1}{n})g(t)}$$ at $t_3$ with initial condition $g(t_4) = \mu$. Number $n$ should be so big that $\min(c, t_0 - t_2, t_2 - t_1) > \frac{2}{n}$.

By the inclusions for the sets $A_n(t)$, $A(t)$ we have $\mu_n(t_2 -\frac{1}{n}) \ge \mu(t_2)$ and $\mu_n(t_1+\frac{1}{n}) \le \mu(t_1)$.  Applying continuity to the solution of the differential equation on the parameters and the initial datum and the fact that $D(t_3, t_4, \mu) \le D(t_3, t_4, \nu)$ if $\mu \le \nu$ we get 
$$D(t_1, t_2, \mu(t_2)) \le \mu(t_1),$$
as required.	

\end{document}